\documentclass{article}
\usepackage{graphicx}
\usepackage{amsmath,amssymb}

\newcommand{\ii}{\mathrm i}

\newcommand{\Cone}{C_{\delta_1}}
\newcommand{\CUone}{\breve{C}_{\delta_1}}
\newcommand{\Ctwo}{\widetilde{C}_{\delta_1}}
\newcommand{\cK}{{\mathcal K}}  

\DeclareMathOperator{\Id}{Id}
\DeclareMathOperator{\sgn}{sgn}
\newcommand{\abs}[1]{|#1|}

\newcommand{\norm}[1]{\| {#1} \|}
\newcommand{\Norm}[1]{\bigg\| {#1} \bigg\|}
\newcommand{\bigPar}[1]{\big(#1\big)}

\newcommand{\BiggPar}[1]{\Bigg(#1\Bigg)}

\newcommand{\operator}[1]{\mathcal{#1}}
\newcommand{\Caputo}[1][\gamma]{\mathcal{D}^{#1}}

\newcommand{\Fourier}{\operator{F}}

\newcommand{\FourierInv}{\Fourier^{-1}}

\newcommand{\rd}{\rho_\delta}
\newcommand{\sd}{s_\delta}
\newcommand{\sdd}{s'_\delta}
\newcommand{\Sd}{S_\delta}

\newcommand{\profile}{\overline u}
\newcommand{\um}{u_-}
\newcommand{\up}{u_+}
\newcommand{\upm}{u_{\pm}}
\newcommand{\wavespeed}{s}
\newcommand{\ShockTriple}{(\um,\up;\wavespeed)}

\newcommand{\Riesz}{D_0^\alpha}  
\newcommand{\RieszFeller}{D_\theta^\alpha}  
\newcommand{\Green}{G_\theta^\alpha}  
\newcommand{\RFsymbol}{\psi_\theta^\alpha}  
\newcommand{\SchwartzTF}{\mathcal{S}} 

\newcommand{\sdiff}[2]{\partial_{#2} {#1}} 
\newcommand{\sdifff}[3]{\partial^{#3}_{#2} {#1}}

\newcommand{\dd}[1][x]{\operatorname{d}\!#1}
\newcommand{\integral}[3]{\int_{#1} {#2} \dd[#3] }
\newcommand{\integrall}[4]{\int_{#1}^{#2} {#3} \dd[#4] }

\newcommand{\C}{\mathbb{C}}
\newcommand{\N}{\mathbb{N}}
\newcommand{\R}{\mathbb{R}}

\newcommand{\SG}{(S_t)_{t\geq 0}}

\newcommand{\Xx}[1]{\quad\text{ #1 }\,}
\newcommand{\XX}[1]{\quad\text{ #1 }\quad}

\usepackage[T1]{fontenc}     
\usepackage{enumitem}        

\usepackage{geometry}
\usepackage{xcolor}

\usepackage{amsthm}
\theoremstyle{plain} 
\newtheorem{theorem}{Theorem}

\newtheorem{lemma}{Lemma}
\newtheorem{proposition}{Proposition}

\theoremstyle{definition}

\newtheorem{remark}{Remark}

\begin{document}

\title{Asymptotic stability of traveling wave solutions for nonlocal viscous conservation laws with explicit decay rates}

\author{Franz Achleitner\thanks{University of Vienna, Faculty of Mathematics, Oskar-Morgenstern-Platz 1, A-1090 Vienna, Austria, {franz.achleitner@univie.ac.at}} and
        Yoshihiro Ueda\thanks{Kobe University, Faculty of Maritime Sciences, 5-1-1 Fukaeminami-machi, Higashinada-ku, Kobe 658-0022, Japan, {ueda@maritime.kobe-u.ac.jp}}}

\maketitle

\begin{abstract}
We consider scalar conservation laws with nonlocal diffusion of Riesz-Feller type such as the fractal Burgers equation.
The existence of traveling wave solutions with monotone decreasing profile has been established recently (in special cases).
We show the local asymptotic stability of these traveling wave solutions in a Sobolev space setting by constructing a Lyapunov functional.
Most importantly, we derive the algebraic-in-time decay of the norm of such perturbations with explicit algebraic-in-time decay rates. 

\smallskip \noindent
\textbf{keywords:} nonlocal evolution equations, Riesz-Feller operator, fractional Laplacian, traveling wave solutions, asymptotic stability, decay rates

\end{abstract}

\section{Introduction} \label{sec:intro}

We consider the evolution of a scalar quantity $u: \R\times (0,\infty) \to U\subset\R$, $(x,t) \mapsto u(x,t)$,
 which is governed by the Cauchy problem 
 \begin{align} 
  	\sdiff{u}{t} + \sdiff{}{x} f(u) &= \RieszFeller u &&\text{for } (x,t)\in\R\times (0,\infty), \label{eq:FCL} \\
   	u(0,x) &= u_0(x) &&\text{for } x\in\R, \nonumber 
 \end{align}
with an initial datum $u_0: \R \to U \subset \R$, a flux function $f: U\subset\R \to \R$
and a Riesz-Feller operator~$\RieszFeller$ for some $1<\alpha \le 2$ and $\abs{\theta}\leq 2-\alpha$.
Equation~\eqref{eq:FCL} models nonlinear transport and nonlocal diffusion of a quantity~$u(x,t)$ in space over time.
The flux function~$f$ is assumed to be smooth and convex as well as to satisfy w.l.o.g. $f(0)=0$. 
The Riesz-Feller operator can be defined as a Fourier multiplier operator,
 see also~\cite{Mainardi+etal:2001}.
 Precisely, the Riesz-Feller operator~$\RieszFeller$ of order~$\alpha$ and skewness~$\theta$ is defined as
 \begin{equation} \label{eq:RF:Fourier}
  \Fourier [\RieszFeller v](k) = \RFsymbol (k) \Fourier [v] (k) \,,\qquad  k \in\R \,,
 \end{equation}
 with symbol
 \begin{equation} \label{eq:RF:symbol}
  \RFsymbol(k) = -|k|^\alpha \exp\left( \ii\, \sgn(k)\, \theta \tfrac{\pi}{2}\right) 
    = -|k|^\alpha \left( \cos( \theta \tfrac{\pi}{2} ) + \ii\, \sgn(k)\, \sin(\theta\tfrac{\pi}{2}) \right)
 \end{equation}
 and parameters $0<\alpha\leq 2$ and $|\theta| \leq \min\{\alpha, 2-\alpha\}$,
  where $\mathcal{F}$ denotes the Fourier transform.

\medskip
\begin{remark}
%
(i) Riesz-Feller operators $\RieszFeller$ with $\theta = 0$ 
 are also known as fractional Laplacians $\Riesz=-(-\sdifff{u}{x}{2})^{\alpha/2}$
 with $0< \alpha \leq 2$ and Fourier symbol $-|k|^\alpha$.
In particular, the Laplacian $D_0^2=\sdifff{}{x}{2}$ is a special case with $\alpha=2$ and $\theta=0$.

(ii) For $0 < \gamma <1$, Riesz-Feller operators $\RieszFeller$  with $\alpha = \gamma$ and $\theta = -\gamma$,
 can be identified with fractional Caputo derivatives of order $0<\gamma<1$: 
 \begin{equation} \label{Caputo}
  -(\Caputo u)(x)= -\frac{1}{\Gamma(1-\gamma)} \integrall{-\infty}{x}{ \frac{u'(y)}{(x-y)^{\gamma}} }{y} \Xx{for} x\in\R \,,
 \end{equation}
 which have Fourier symbol $-(-\ii k)^\gamma$.
 The symbol $(-\ii k )^\gamma$ is multi-valued, however (only) the choice
  $(-\ii k)^\gamma = \left(|k|\exp(-\ii\sgn(k)\, \tfrac\pi2)\right)^\gamma=|k|^\gamma \exp(-\ii\sgn(k)\, \gamma\tfrac\pi2)$
  yields a causal operator. For details, see \cite{Kempfle+etal:2002}.
 Moreover, its derivative $\sdiff{}{x} (\Caputo u)$ is a Riesz-Feller operator with $\alpha = 1+\gamma$ and $\theta = 2 -\alpha$.
\end{remark}

Taking $\alpha = 2$ and $\theta = 0$ in \eqref{eq:FCL}, we formally obtain a classical viscous conservation law:
\begin{equation} \label{eq:VCL}
  \sdiff{u}{t} + \sdiff{}{x} f(u) = \sdifff{u}{x}{2} \Xx{for} (x,t)\in\R\times (0,\infty).
 \end{equation}
The existence and asymptotic stability of traveling wave solutions of equation~\eqref{eq:VCL} has been studied thoroughly.
A first example of equation~\eqref{eq:FCL} with nonlocal diffusion is    
\begin{equation} \label{eq:FCL:FL} 
 \partial_t u + \partial_x f(u) =  \Riesz u \Xx{for} (x,t)\in\R\times (0,\infty)\,,
\end{equation}
with a fractional Laplacian $\Riesz$, $0<\alpha\leq 2$, which has been studied e.g. in~\cite{Biler+etal:1998,Droniou+etal:2003}. 
For $1<\alpha\leq 2$, the Cauchy problem for~\eqref{eq:FCL:FL} with $f\in C^\infty(\R)$ and essentially bounded initial data
 has a global-in-time mild solution which becomes smooth for positive times,
 see~\cite{Droniou+etal:2003} and its extension to~\eqref{eq:FCL} in~\cite{Achleitner+etal:2012}.

Other examples of equation~\eqref{eq:FCL} with nonlocal diffusion appear in viscoelasticity~\cite{Sugimoto+Kakutani:1985}
 and fluid dynamics~\cite{Kluwick+etal:2010}.
In particular,  
 \begin{equation} \label{eq:FCL:onesided}
  \sdiff{u}{t} + \sdiff{}{x} f(u) = \sdiff{}{x} \Caputo u \Xx{for} (x,t)\in\R\times (0,\infty)\,,
 \end{equation}
 with $0<\gamma<1$ is used as a model for the far-field behavior of uni-directional viscoelastic waves~\cite{Sugimoto+Kakutani:1985},
 and derived as a model for the internal structure of hydraulic jumps in near-critical single-layer flows~\cite{Kluwick+etal:2010}. 
Moreover the nonlocal operator $\Caputo[1/3]$ appears in Fowler's equation
\begin{equation} \label{eq:Fowler}
  \partial_t u + \partial_x u^2= \partial_x^2 u-\partial_x \Caputo[1/3] u \,,
\end{equation}
which models the uni-directional evolution of sand dune profiles~\cite{Fowler:2002}.
In the theory of water waves similar models $\partial_t u + \partial_x u^2 =\mathcal{N}[u]$ with different (nonlocal) Fourier multiplier operators $\mathcal{N}$ are studied,
 see the book~\cite{Naumkin+Shishmarev:1994} and references therein.

\medskip
%
To explain our main results, we introduce traveling wave solutions for equation~\eqref{eq:FCL}.
Traveling wave solutions (TWS) are of the form $u(x,t)=\overline u(\xi)$ for some profile~$\profile$
 with $\xi  = x - \wavespeed t$ and (constant) wave speed~$\wavespeed\in\R$.
We are interested in TWS with profiles~$\profile$ connecting distinct endstates~$\upm$ such that
\begin{equation}\label{endpoint}
	\lim_{x \to \pm \infty} \profile(x) = u_{\pm} \ .
\end{equation}
Using this ansatz in equation \eqref{eq:FCL} and assumption \eqref{endpoint},
 we find that the wave speed $\wavespeed$ has to satisfy the Rankine-Hugoniot condition
 \begin{equation} \label{RH}
  s = \frac{f(\up) - f(\um)}{\up -\um}.
 \end{equation}
Here, an extension of Riesz-Feller operators to non-integrable functions is needed, see Appendix~\ref{app:RF}.
Due to translational invariance of equation~\eqref{eq:FCL},
 traveling wave solutions are only unique up to a shift.

For classical viscous conservation laws~\eqref{eq:VCL},
 the profile of a TWS satisfies an ordinary differential equation $\profile' = f(\profile)-s\profile - (f(\um)-s\um)$.
In fact, TWS exist only for parameters $\ShockTriple$ satisfying~\eqref{RH} and $\up<\um$.
In case of equation~\eqref{eq:FCL:onesided},
 the existence and asymptotic stability (without decay rates) of traveling wave solutions
 for parameters $\ShockTriple$ satisfying~\eqref{RH} and $\up<\um$
 has been shown~\cite{Achleitner+Hittmeir+Schmeiser:2011,Cuesta+Achleitner:2017}.
Here, a profile satisfies a fractional differential equation $\Caputo \profile = f(\profile)-s\profile - (f(\um)-s\um)$.
The proof of existence relies on the causality of the Caputo derivative~$\Caputo$,
 i.e. to evaluate $\Caputo\profile$ at $x$ the profile $\profile$ on $(-\infty,x)$ is needed.
In contrast, the profile for a TWS of a nonlocal conservation law~\eqref{eq:FCL:FL} for $1<\alpha<2$ has to satisfy 
 \[ \Riesz \profile (x)
      =\integral{\R}{ \frac{\profile(x+\xi)-\profile(x)-\profile'(x)\,\xi}{\xi^{1+\alpha}} }{\xi}
      = \sdiff{}{x} \big(f(\profile)-s\profile - (f(\um)-s\um) \ . 
 \]
Thus $\Riesz \profile (x)$ depends on the whole profile~$\profile$.
For fractal Burgers equation, i.e. equation~\eqref{eq:FCL:FL} with $1<\alpha<2$ and Burgers flux function $f(u)=u^2$,
 the existence of traveling wave solutions has been proven recently~\cite{Chmaj:2014}.
The idea is to approximate the operators $\Riesz$ by convolution operators $\cK_\epsilon[u] =K_\epsilon\ast u -u$
 for suitable convolution kernels $K_\epsilon\in L^1(\R)$.
The existence of TWS for the approximate equations is known
 and the TWS is established as the limit of this family. 
It is conceivable to use this approach to prove the existence of traveling wave solutions 
 in the general case~\eqref{eq:FCL} for convex flux functions~$f$ with $1<\alpha<2$ and $|\theta|\leq 2-\alpha$.

For fractal Burgers equation~\eqref{eq:FCL:FL} results in the complementary cases 
 $\alpha\in (0,1)$ and/or $\um\leq\up$ are also available:
For example, for $\alpha\in (0,1)$ and~\eqref{endpoint} no traveling wave solutions of~\eqref{eq:FCL:FL}
 with smooth profile exists~\cite{Biler+etal:1998}.
Whereas under the assumption $\um<\up$
 the solution of~\eqref{eq:FCL:FL} converges as $t\to\infty$ to a rarefaction wave of the underlying Burgers equation if $\alpha\in (1,2)$
 and to a self-similar solution if $\alpha=1$; see~\cite{Karch+Miao+Xu:2008} and~\cite{Alibaud+Imbert+Karch:2010}, respectively.

\medskip
The asymptotic stability of traveling wave solutions of classical viscous conservation laws~\eqref{eq:VCL} has been studied thoroughly.
At first, historically, Il'in and Oleinik \cite{IO60} proved the asymptotic stability 
of nonlinear waves for viscous conservation laws \eqref{eq:VCL} by  
making use of the maximum principle for linear parabolic equations.
For Burgers' equation, i.e.~equation~\eqref{eq:VCL} with Burgers' flux function~$f(u) = u^2$, Nishihara \cite{N85} obtained the decay estimates toward 
traveling wave solutions by making use of the explicit solution formula.
And, Kawashima and Matsumura \cite{Kawashima+Matsumura:1985} generalized Nishihara's time decay result to a class of viscous conservation laws.
They considered weighted $L^2$ spaces and used a weighted energy method.
Furthermore, Kawashima, Nishibata and Nishikawa \cite{Kawashima+etal:2004} extended the $L^2$ energy method to general $L^p$ spaces.
Their techniques have been applied to a model system for compressible viscous gas in \cite{MN85} and a hyperbolic system with relaxation in \cite{U09}.

Assuming the existence of a traveling wave solution of~\eqref{eq:FCL} with monotone decreasing profile,
 we show that asymptotic stability of a traveling wave solution in a Sobolev space setting follows from a standard Lyapunov functional argument:
To investigate the stability of the traveling wave solution with profile~$\profile$, 
 we consider initial data $u_0$ such that $u_0 - \profile$ is integrable
 and determine the unique shift $x_0$ which yields 
$
\integrall{-\infty}{\infty}{ (u_0(\xi) - \profile(\xi + x_0)) }{\xi} = 0.
$
Moreover,
 we restrict the domain of initial data $u_0$ further such that 
$
W_0(\xi)=\integrall{-\infty}{\xi}{ (u_0(\eta)-\profile(\eta))}{\eta}
$ 
exists (using a suitable shifted profile~$\profile$) and satisfies $W_0 \in H^2$.
(For details, we refer to \cite{U09}.)
More precisely, we can derive the following theorem.

\begin{theorem}\label{theorem:AS}
Suppose $1<\alpha \le 2$ and $\theta\leq\min\{\alpha,2-\alpha\}$.
Let the flux function $f\in C^2(\R)$ be convex
 and let $u(x,t)=\profile(x- \wavespeed t)$ be a 
 traveling wave solution of~\eqref{eq:FCL}
 with monotone decreasing profile~$\profile$.
Let $u_0$ be an initial datum for~\eqref{eq:FCL}
 such that 
 $W_0(\xi)=\integrall{-\infty}{\xi}{ (u_0(\eta)-\profile(\eta)) }{\eta}$
 satisfies $W_0\in H^2(\mathbb{R})$.
Then there exists a positive constant $\delta_0$ such that if $\norm{W_0}_{H^2} \le \delta_0$,
then the Cauchy problem \eqref{eq:FCL}
has a unique global solution
converging to the traveling wave in the sense that
$$
\| (u - \profile)(t)\|_{L^\infty} \longrightarrow 0  \qquad \text{for} \quad t \to \infty. 
$$
\end{theorem}

The proof of Theorem~\ref{theorem:AS} for the general equation~\eqref{eq:FCL}
 is similar to the one of~\cite[Theorem 4]{Achleitner+Hittmeir+Schmeiser:2011} for the special case~\eqref{eq:FCL:onesided} without decay rates.

Our main result is to prove the asymptotic stability with algebraic-in-time decay rate
 for traveling wave solutions of~\eqref{eq:FCL} with monotone decreasing profiles.

\begin{theorem}\label{theorem:CR}
Suppose the same assumptions as in Theorem~\ref{theorem:AS} hold
 and $f\in C^\infty(\R)$.
For all $W_0\in W^{1,\infty}(\R)\cap W^{1,1}(\R)$,
 the Cauchy problem~\eqref{eq:FCL} has a unique global solution.
Moreover, there exists a positive constant $\delta_1$ such that
 if $\norm{W_0}_{W^{1,1}} \le \delta_1$
 then the unique global solution~$u$ satisfies
 \begin{equation}\label{decay}
  \|(u - \profile)(t)\|_{L^2} \leq C E_1 (1+t)^{-1/(2\alpha)} 
 \end{equation}
for $t \ge 0$, where $E_1 := \|W_0\|_{H^1} + \|W_0\|_{W^{1,1}}$ and $C$ is a constant which is independent of time~$t$.
\end{theorem}


\begin{remark}
We employ sharp interpolation inequalities in Sobolev spaces to derive~\eqref{decay}.
In this way optimal decay estimates for the asymptotic stability of viscous rarefaction waves
 in scalar viscous conservation laws~\eqref{eq:VCL} have been derived in~\cite{HUK09}.
\end{remark}

\begin{remark}
We want to explain the functional setting in Theorem~\ref{theorem:CR}:
We considered the function spaces
  $H^2(\R)\cap W^{2,1}(\R) \subset W^{1,\infty}(\R)\cap W^{1,1}(\R) \subset H^1(\R)\cap W^{1,1}(\R)$
	in variants of Theorem~\ref{theorem:CR}.
 The choice $H^1(\R)\cap W^{1,1}(\R)$ leads to the restriction $\alpha\in(3/2,2)$
  if we use an estimate of the nonlinearity like Dix~\cite{Dix:1992,Dix:1996}
	to establish the existence of solutions for the Cauchy problem.
Assuming higher regularity of the initial data removes the need for this restriction:	
 Under the assumptions of Theorem~\ref{theorem:AS} 
 with $W_0\in H^2(\R)\cap W^{2,1}(\R)$, 
 the solution constructed in Theorem~\ref{theorem:AS} satisfies 
 \begin{equation*}
  \|(u - \profile)(t)\|_{H^1} \leq C \widetilde{E}_1 (1+t)^{-1/(2\alpha)} 
 \end{equation*}
 for $t \ge 0$, where $\widetilde{E}_1 := \|W_0\|_{H^2} + \|W_0\|_{W^{2,1}}$ 
 and a constant $C$ independent of time $t$.
Our choice $W_0 \in W^{1,\infty}(\R)\cap W^{1,1}(\R)$ in Theorem~\ref{theorem:CR}
 leads to the technical assumption $f\in C^\infty(\R)$,
 since we use a result on the existence of global-in-time solutions for the Cauchy problem
 with essentially bounded initial data~\cite{Droniou+etal:2003,Achleitner+etal:2012}.
The assumption $f\in C^2(\R)$ in Theorem~\ref{theorem:AS} could be retained
 by aiming for less regularity in their approach.
\end{remark}

\medskip
Unfortunately it is difficult to apply the weighted energy method in \cite{Kawashima+Matsumura:1985} to our problem (to derive the convergence rate).
Instead of this method, we employ another technique which focuses on the interpolation property in Sobolev space.
For example, this argument is utilized in \cite{HUK09}.

The contents of this paper are as follows. 
In Section~\ref{sec:reformulation}, we reformulate our problem and consider the well-posedness of the new one.
In Section~\ref{sec:AS}, we derive the asymptotic stability result 
by uniform energy estimates as {\it a-priori} estimates of solutions in the Sobolev space $H^2$.
Furthermore, our main result on the asymptotic stability with explicit algebraic decay rate in Theorem~\ref{theorem:CR} is proved in Section~\ref{sec:Rates}, 
by using the energy method with an $L^2$--$L^1$ interpolation argument. 
In Appendix~\ref{app:RF}, we collect results on the singular integral representation of Riesz-Feller operators.

 \bigskip
 \noindent
 {\bf Notation.}
Before closing this section, we give some notations used in this paper.
We define the Fourier transform for $v\in\SchwartzTF$ in the Schwartz space $\SchwartzTF$ as 
 \begin{align*}
  \hat{v}(k) = \Fourier [v](k) &:= \integral{\R}{ e^{-\ii k x} v(x) }{x} \Xx{for} k \in \R \,,
 \intertext{and the inverse Fourier transform as}
  \FourierInv [v](x) &:= \frac{1}{2\pi} \integral{\R}{ e^{\ii k x} v(k) }{k} \Xx{for} x\in\R \,.
 \end{align*}
 The Fourier transform and its inverse are linear operators
 and $\Fourier$ and $\FourierInv$ will denote also their respective extensions to $L^2(\R)$.
 
For $1 \le p \le \infty$, we denote by $L^p=L^p(\R)$ the usual Lebesgue space over $\R$ 
with norm $\| \cdot \|_{L^p}$, and $W^{s,p}=W^{s,p}(\R)$ the usual Sobolev space over $\R$ 
with norm $\| \cdot \|_{W^{s,p}}$.
Using the short-hand notation $H^{s}(\R) := W^{s,2}(\R)$ with norm $\| \cdot \|_{H^{s}}$.
Moreover, we set $\norm{W(t)}_{W^{1,\infty}} =\max\{\norm{W(t)}_{L^\infty},\ \norm{\sdiff{W}{\xi}(t)}_{L^\infty}\}$
 and its analog in case of $\norm{W(t)}_{W^{\ell,\infty}}$ for all $\ell\in\N$.
Finally, for nonnegative integer~$\ell$, $C^{\ell}(I;X)$ (respectively $C^{\ell}_b(I;X)$) denotes the space of 
$\ell$-times continuously differentiable functions (respectively with bounded derivatives)
on the interval $I$ with values in the Banach space $X$. 

The constants in our estimates may change their value from line to line.


\section{Reformulation for the problem} \label{sec:reformulation}

In the special case~\eqref{eq:FCL:onesided},
 the existence and asymptotic stability of traveling wave solutions $u(x,t)=\profile(x-\wavespeed t)$ 
 with monotone decreasing profile~$\profile$ has been proven without rates of decay~\cite{Achleitner+Hittmeir+Schmeiser:2011,Cuesta+Achleitner:2017}.
However, assuming in the general case~\eqref{eq:FCL} the existence of a traveling wave solution $u(x,t)=\profile(x-\wavespeed t)$
 with monotone decreasing profile~$\profile$,
 then the proof of asymptotic stability generalizes with obvious modifications:

To prove the asymptotic stability of a traveling wave solution~$\profile$ of~\eqref{eq:FCL},
 one can follow the standard approach called the anti-derivative method introduced in \cite{Kawashima+Matsumura:1985} 
 for viscous conservation laws.
It is convenient to cast~\eqref{eq:FCL} in a moving coordinate frame $(x,t)\mapsto (\xi,t)$, such that 
\begin{equation} \label{CLND+MCF}
  \partial_t u + \partial_\xi (f(u)-\wavespeed u) = \RieszFeller u \,,
\end{equation}
and $\profile$ is a stationary solution of~\eqref{CLND+MCF}.
The Cauchy problem for~\eqref{CLND+MCF} with initial datum $u_0$ governs the evolution of $u_0$.
If its solution $u$ is considered as a perturbation of the traveling wave solution $\profile$,
 then this perturbation $U(\xi,t) := u(\xi,t) - \profile(\xi)$ satisfies the Cauchy problem 
 \begin{equation} \label{CP:U}
 	\begin{split}
 \partial_t U + \partial_\xi (f(\profile+U)-f(\profile)) -\wavespeed \partial_\xi U &= \RieszFeller U, \\
 U(\xi,0) &=U_0(\xi),
	\end{split}
 \end{equation}
 where $U_0(\xi):=u_0(\xi)-\profile(\xi)$.
 %
%
To obtain the desired result, we try to construct the $L^2$-energy estimate for $U$ by employing the energy method.
However, because of the decreasing property of traveling wave solutions, it is hard to construct the $L^2$-energy estimate.
To overcome this difficulty, we apply the anti-derivative method.

Precisely, we introduce the new function $W(\xi,t)$ which satisfies $\partial_\xi W = U$.
Then we can formally rewrite \eqref{CP:U} as
 \begin{equation}\label{CP:W}
 	\begin{split}
 \sdiff{W}{t} + f(\profile+\sdiff{W}{\xi})-f(\profile)-\wavespeed \sdiff{W}{\xi} &= \RieszFeller W,\\
 W(\xi,0) &= W_0(\xi).
 	\end{split}
 \end{equation}
If a global-in-time solution of \eqref{CP:W} with $W_0(\xi) = \integrall{-\infty}{\xi}{ U_0(\eta) }{\eta}$ is sufficiently smooth,
 then its derivative $\partial_\xi W$ satisfies Cauchy problem \eqref{CP:U}.
Therefore, we try to construct a global-in-time solution of \eqref{CP:W}, instead of \eqref{CP:U}.
For this purpose, we discuss the well-posedness of problem \eqref{CP:W} in this section.


The well-posedness of the Cauchy problem for~\eqref{CP:W},
will follow from a contraction argument.
Assuming $f(u)=u^2$ and $\alpha >3/2$ allows to estimate the nonlinearity
 in the fashion of Dix~\cite{Dix:1992,Dix:1996} implying the well-posedness in $H^1$.
For general flux functions and $\alpha \in (1,2]$,
 we have to require more regularity of the initial data,
 {e.g.} $W_0\in H^2$. 

\begin{proposition} \label{prop:CP:H2}
Let $f\in C^2(\R)$, $1<\alpha \le 2$ and $\abs{\theta}\leq \min\{\alpha,2-\alpha\} =2-\alpha$.
Suppose $M$ is an arbitrary positive constant and
 suppose $W_0\in H^2(\R)$ such that $\|W_0\|_{H^2} \le M$.
Then there exists a positive constant $T$, which depends on $M$, such that
 the Cauchy problem~\eqref{CP:W} has a unique mild solution 
 $W \in C([0,T];H^2)$ with $\|W(t)\|_{H^2} \le 2M$ for $t \in [0,T]$.
\end{proposition}

To prove Proposition~\ref{prop:CP:H2}, we first present some properties 
of the fundamental solution of $\sdiff{u}{t}= \RieszFeller u$.

\begin{lemma}[{\cite[Lemma 2.1]{Achleitner+Kuehn:2015}}] \label{lem:SSPM} 
For $1<\alpha\leq 2$ and $\abs{\theta} \leq \min\{\alpha,2-\alpha\} =2-\alpha$,
 $\Green(x,t) := \mathcal{F}^{-1}[e^{t \psi^\alpha_\theta (\cdot)}](x)$ with $\psi^\alpha_\theta$ defined in~\eqref{eq:RF:symbol}
 is the fundamental solution of $\sdiff{u}{t} =\RieszFeller u$.
Moreover, $\Green$ satisfies for all $(x,t)\in \R \times (0,\infty)$ the properties
  \begin{enumerate}[label=(G\arabic*)] 
    \item \label{K:SFDE:prop0} $\Green(x,t)\geq 0$,
    \item \label{K:SFDE:prop1} $\Green(x,t)=t^{-1/\alpha} \Green (x t^{-1/\alpha},1)$, 
    \item \label{K:SFDE:prop2} $\norm{\Green(\cdot,t)}_{L^1(\R)}=1$,
    \item \label{K:SFDE:prop3} $\Green(\cdot,s)\ast \Green(\cdot,t) = \Green(\cdot,s+t)$ for all $s,t\in(0,\infty)$,
    \item \label{K:SFDE:prop4} $\norm{\Green(\cdot,t)}_{L^p(\R)}\leq \norm{\Green(\cdot,1)}_{L^p(\R)} 
    t^{-\frac1\alpha (1-\frac1p)}$ for all $1\leq p <\infty$,
    \item \label{K:SFDE:prop5} $\Green\in C^\infty_0(\R\times (0,\infty))$,
   \item \label{K:SFDE:prop7} For all $t>0$, there exists a constant $\cK$ such that $\norm{\sdiff{G}{x} (\cdot,t)}_{L^1(\R)} \leq \cK t^{-1/\alpha}$.
  \end{enumerate}
\end{lemma}

Due to the properties of $\Green$,
 it is easy to show that $\RieszFeller$ generates a semigroup.
 \begin{lemma} \label{lem:SFDE:semigroup}
  For $1<\alpha\leq 2$, $\abs{\theta} \leq \min\{\alpha,2-\alpha\} =2-\alpha$,
  the Riesz-Feller operator~$\RieszFeller$ generates a strongly continuous, convolution semigroup 
  \[
   S_t: L^p(\R) \to L^p(\R)\,, \quad u_0 \mapsto S_t u_0 = \Green(\cdot,t)\ast u_0  
  \]
	with $\Green$ defined in Lemma~\ref{lem:SSPM}.
  Moreover, the semigroup satisfies the dispersion property for $u\in L^1(\R)$
  \begin{equation} \label{prop:SG:dispersion}
    \norm{S_t u}_{L^p(\R)} \leq C_p\; t^{-\frac1\alpha (1-\frac1p)} \norm{u}_{L^1(\R)} \quad 
  \end{equation}
  for all $1\leq p<\infty$ and some $C_p>0$.
 \end{lemma}

\begin{proof}
Due to~\ref{K:SFDE:prop2} and Young's inequality for convolutions,
 \[ \norm{S_t u}_{L^p} \leq \norm{\Green(\cdot ,t)}_{L^1} \norm{u}_{L^p} = \norm{u}_{L^p} \]
 for all $u\in L^p(\R^n)$.
Therefore $S_t:L^p(\R)\rightarrow L^p(\R)$ are well-defined bounded linear operators for all~$t\geq 0$. 
$\SG$ is a semigroup,
 since $S_{t+s}=S_t S_s$ for all $s,t\geq 0$ holds due to \ref{K:SFDE:prop3} and $S_0:=\Id$.
Strong continuity of~$\SG$
 follows from a standard result about convolutions~\cite[p.64]{Lieb+Loss:2001} and~\ref{K:SFDE:prop1}.
The dispersion property
 \[
   \forall 1\leq p<\infty \quad \exists C_p>0 \,: \quad 
   \norm{S_t u}_{L^p(\R)} \leq C_p\, t^{-\frac1\alpha (1-\frac1p)} \norm{u}_{L^1(\R)} \quad
   \forall u\in L^1(\R)
 \]
 can be proved using \ref{K:SFDE:prop4} and Young's inequality \cite[p.98-99]{Lieb+Loss:2001}.
\end{proof}

\begin{lemma} \label{lem:ker}
Let $1<\alpha \le 2$ and $\abs{\theta}\leq \min\{\alpha,2-\alpha\}$. 
The fundamental solution~$\Green$ defined in Lemma~\ref{lem:SSPM} satisfies for all $\ell\in\N_0$ and $0 \le r \le \ell$ the following estimates:
\begin{equation} \label{ker_est} 
 \| \partial_x^{\ell} \big( \Green (t) \ast \phi\big)\|_{L^2} \le C t^{- (\ell -r)/\alpha} \|\partial_x^r \phi\|_{L^2}, \qquad t>0\,, 
\end{equation}
where $C$ is a certain positive constant.
If $r = \ell$, then inequality \eqref{ker_est} with $C=1$ is optimal.
\end{lemma}

\begin{proof} 
By using Plancherel's theorem, we compute that
\begin{multline*}
	\| \partial_x^{\ell} \big( \Green (t) \ast \phi\big)\|_{L^2}
	 = \|(\ii k)^{\ell} e^{t \psi^\alpha_\theta (k)} \hat{\phi}\|_{L^2} \\
	 \le \|(\ii k)^{\ell - r}e^{t \psi^\alpha_\theta (k)}\|_{L^\infty} \|(\ii k)^{r} \hat{\phi}\|_{L^2}
	 \le C t^{- (\ell -r)/\alpha} \|\partial_x^{r}\phi\|_{L^2} ;
\end{multline*}
since
$\|(\ii k)^{\ell - r}e^{t \psi^\alpha_\theta (k)}\|_{L^\infty} 
	 = \sup_{k \in \mathbb{R}} |k|^{\ell -r}e^{-t |k|^\alpha \cos (\theta \pi/2)}
	 \le C t^{- (\ell -r)/\alpha}$,
due to the positivity of $\cos (\theta \pi/2)$ under the assumption in Lemma~\ref{lem:ker}.
If $r = \ell$, then we obtain
$\| \partial_x^{\ell} \big( \Green (t) \ast \phi\big)\|_{L^2} 
 \leq \|G^\alpha_\theta\|_{L^1} \|\partial_x^{\ell}\phi\|_{L^2} 
 = \|\partial_x^{\ell}\phi\|_{L^2}$,
by using the fact that $\Green$ is a non-negative integrable function with mass one. 
\end{proof}

\begin{lemma} \label{lem:conti}
Suppose that the same assumption as in Lemma~\ref{lem:ker} holds, and $\phi \in H^\sigma$ for $\sigma \ge 0$. 
Then the fundamental solution 
satisfies
$\Green \ast \phi \in C([0,\infty);H^\sigma)$.
\end{lemma}

\begin{proof} 
For arbitrary constants $t_1, t_2 \in [0,\infty)$, 
we have
$$
  \norm{\Green(t_1)\ast \phi - \Green(t_2) \ast \phi}^2_{H^\sigma}
     \le \integral{\R}{ (1+|k|)^{2\sigma}|e^{t_1\RFsymbol(k)} - e^{t_2\RFsymbol(k)}|^2|\hat{\phi}(k)|^2 }{k},
$$
where the integral is bounded by $4 \|\phi\|^2_{H^\sigma}$.
Thus, the Dominated Convergence Theorem allows to pass to the limit under the integral sign,
 which completes the proof.
\end{proof}

\bigskip
\begin{proof}[Proof of Proposition~\ref{prop:CP:H2}] 
Using the fundamental solution~$\Green$ of the linear evolution equation $\sdiff{u}{t}= \RieszFeller u$,
the mild formulation of~\eqref{CP:W} reads
\begin{equation} \label{mildform:W}
 W(t) = \Green(t)\ast W_0 
  - \integrall{0}{t}{ \Green(t-\tau)\ast F(\profile,\sdiff{W}{\xi}) }{\tau},
\end{equation} 
where $F(\profile,\sdiff{W}{\xi}) := f(\profile+\sdiff{W}{\xi})-f(\profile)-\wavespeed \sdiff{W}{\xi}$.
To employ a fix point argument,
 we consider the mapping $\operator{G}[W]$ defined by
 \begin{equation}\label{map:G}
 \operator{G}[W](t) := \Green(t)\ast W_0 
  - \integrall{0}{t}{ \Green(t-\tau)\ast F(\profile,\sdiff{W}{\xi})}{\tau},
 \end{equation}
 on the Banach space $X:= C([0,T];H^2)$ 
 with norm $\norm{W}_X:=\sup_{t\in[0,T]} \norm{W(t)}_{H^2}$. 
Then we show that $\operator{G}$ 
 is a contraction mapping on a closed convex subset $S_R$ of $X$,
 where $S_R := \{W \in X ; \|W\|_X \le R\}$ for some parameter $R > 0$ which will be determined later.

  

Due to a Sobolev embedding, $\|W\|_{X} \le R$ implies that $\|W(t)\|_{W^{1,\infty}} \le R$ for $t \in [0,T]$.
Thus, if $\|W\|_{X} \le R$ and $\ell = 0, 1$, then we compute that 
 \begin{equation*}
 \begin{split}
  \|\partial_\xi^{\ell}(\operator{G}[W] &- \operator{G}[V])(t)\|_{L^2} \\
&\le \integrall{0}{t}{ \|\partial_\xi^{\ell} \Green(t-\tau)\ast \{F(\profile,\sdiff{W}{\xi}) - F(\profile,\sdiff{V}{\xi})\}\|_{L^2} }{\tau}  \\
&\le C \integrall{0}{t}{  (t-\tau)^{-\ell/\alpha}\|\{F(\profile,\sdiff{W}{\xi}) - F(\profile,\sdiff{V}{\xi})\}(\tau)\|_{L^2} }{\tau}  \\
&\le C(C(R) + |s|) \integrall{0}{t}{ (t-\tau)^{-\ell/\alpha}\|\partial_\xi(W-V)(\tau)\|_{L^2} }{\tau} \\
&\le C_\ell(R)\ t^{1-\ell/\alpha}\ \|W-V\|_{X}
 \end{split}
 \end{equation*}
 where we used Lemma~\ref{lem:ker} and the identity
 \begin{multline*}
  F(\profile,\sdiff{W}{\xi}) - F(\profile,\sdiff{V}{\xi}) 
  = f(\profile + \sdiff{W}{\xi}) - f(\profile + \sdiff{V}{\xi})  -s \partial_\xi(W - V) \\
  = \integrall{0}{1}{\big[f'(\profile + \sigma \partial_\xi W +(1-\sigma)\partial_\xi V))-s\big]\, \partial_\xi(W - V) }{\sigma}.
 \end{multline*}
Similarly, we can calculate that
 \begin{equation*}
 \begin{split}
\|\partial_\xi^2 (\operator{G}[W] &- \operator{G}[V])(t)\|_{L^2} \\
&\le \integrall{0}{t}{ \|\partial_\xi \Green(t-\tau)\ast \partial_\xi \{F(\profile,\sdiff{W}{\xi}) - F(\profile,\sdiff{V}{\xi})\}\|_{L^2} }{\tau}  \\
&\le C \integrall{0}{t}{ (t-\tau)^{-1/\alpha}\|\partial_\xi \{F(\profile,\sdiff{W}{\xi}) - F(\profile,\sdiff{V}{\xi})\}(\tau)\|_{L^2} }{\tau}  \\
&\le C(C(R) + |s|) \integrall{0}{t}{ (t-\tau)^{-1/\alpha} \| (W-V)(\tau)\|_{H^2} }{\tau}  \\
&\le C_2(R)\ t^{1-1/\alpha}\ \|W-V\|_{X}.
 \end{split}
 \end{equation*}
Combining the above estimates, we obtain 
 \begin{equation*}
 \begin{split}
  \|\operator{G}[W] - \operator{G}[V]\|_{X} 
\le \{C_0(R)T^{1/\alpha} + C_1(R) + C_2(R)\}T^{1 - 1/\alpha} \|W-V\|_{X}.
 \end{split}
 \end{equation*}
Therefore, letting $T = \min \{ 1, (2C_*(R))^{-\alpha/(\alpha-1)} \}$, we deduce
 \begin{equation}\label{W-V}
  \|\operator{G}[W] - \operator{G}[V]\|_{X} \le \frac12 \|W-V\|_{X},
 \end{equation}
 where $C_*(R) := C_0(R) + C_1(R) + C_2(R)$.
 On the other hand, letting $V \equiv 0$ in \eqref{W-V}, we get
 \begin{equation*}
  \|\operator{G}[W] \|_{X} \le  \|\operator{G}[0]\|_X +  \frac12 \|W\|_{X}
  \le  \|W_0\|_{H^2} +  \frac12 \|W\|_{X}
  \le  M +  \frac12 R,
 \end{equation*}
 where we used \eqref{ker_est} with $\ell = r$.
 Therefore, choosing $R = 2M$, we obtain 
 $\|\operator{G}[W] \|_{X} \le  2 M$.
 
Finally we discuss the continuity of $\operator{G} [W]$ in time $t$.
It follows from the continuity at time~$0$
 and the semigroup property~\ref{K:SFDE:prop3} of $\Green$.
Due to Lemma~\ref{lem:conti}, for $W_0\in H^\sigma(\R)$ with $\sigma\geq 0$,
 the convergence $\lim_{t\searrow 0} \Green(\cdot,t)\ast W_0 = W_0$ in $H^\sigma$ holds. 
Moreover, for $t\in[0,T]$ and $s\geq 0$ the identity
 \begin{align*} 
  \operator{G} [W](s+t)
  &= \Green(\cdot,s+t)\ast W_0 (x)
    - \integrall{0}{s+t}{ \Green(\cdot,s+t-\tau)\ast F(\profile,\sdiff{W}{\xi}(\tau)) }{\tau} \\
  &= \Green(\cdot,s)\ast \BiggPar{ \operator{G} [W](t)
    - \integrall{t}{s+t}{ \Green(\cdot,t-\tau)\ast F(\profile,\sdiff{W}{\xi}(\tau)) }{\tau} }
 \end{align*}
 holds, where the last integral converges to zero for $s\to 0$.
Thus, for~$t_1, t_2 \in [0,T]$ with $t_1<t_2$ (without loss of generality), we have 
\begin{multline} \label{conti_G}
 \operator{G}[W](t_1) - \operator{G}[W](t_2)
   = \operator{G}[W](t_1) - \operator{G}[W]((t_2-t_1) +t_1) \\
   = \operator{G}[W](t_1) - \Green(\cdot,t_2 -t_1)\ast \BiggPar{ \operator{G} [W](t_1)
    - \integrall{t_1}{t_2}{ \Green(\cdot,t_1-\tau)\ast F(\profile,\sdiff{W}{\xi}(\tau)) }{\tau} }.
\end{multline}
Therefore, by the fact that $W_0 \in H^2$, $W \in X$ and Lemma~\ref{lem:conti}, 
we find that the right hand side of \eqref{conti_G} tends to zero in $H^2$ as $t_1 \to t_2$.
Hence, we deduce the continuity of $\operator{G}[W]$ in $t$  
 and that $\operator{G}[W] \in S_{2M}$ for $W \in S_{2M}$.
%

Consequently, we conclude that there exist $T=T(M)$ such that $\operator{G}$ 
 is a contraction mapping of $S_{2M}$.
This means that the mapping $\operator{G}$ admits a unique fixed point $W$ in $S_{2M}$,
such that $W = \operator{G}[W]$. 
Hence the proof of Proposition~\ref{prop:CP:H2} is complete.
\end{proof}

%


\section{Asymptotic stability of traveling waves} \label{sec:AS}

In this section, we consider the asymptotic stability of traveling wave solutions with monotone decreasing profile in~\eqref{eq:FCL}.
To this end we derive the existence of global-in-time solutions for evolution equation~\eqref{CP:W}
 and that these perturbations decay. 
Precisely we prove the following theorem.

\begin{theorem}\label{theorem:ASW}
Suppose that the same assumptions as in Theorem~\ref{theorem:AS} hold.
Then the Cauchy problem \eqref{CP:W}
has a unique global solution $W(\xi,t)$ satisfying $W \in C([0,\infty);H^2) \cap C^1([0,\infty);H^1)$ and
\begin{equation}\label{energy_est}
  \|W(t)\|_{H^2}^2 
+ C \sum_{\ell = 0}^2 \integrall{0}{t}{ \|W(\tau)\|^2_{\dot{H}^{\alpha/2 + \ell}} }{\tau}
    - \integrall{0}{t}{ \integral{\R}{ f''(\profile) \profile' W^2 }{\xi} }{\tau}
      \le \|W_0\|_{H^2}^2
\end{equation}
for some positive constant $C$ and for all $t \ge 0$. 
Furthermore, the solution $W(\xi,t)$ converges to zero 
in the sense that
\begin{equation}\label{asyW}
\|W(t)\|_{W^{1,\infty}} \longrightarrow 0  \qquad \text{for} \quad t \to \infty. 
\end{equation}
\end{theorem}

 We note that the third integral of the left hand side in \eqref{energy_est} is non-negative,
 since the flux function $f\in C^2$ is convex such that $f''\geq 0$
 and the profile $\profile$ is monotone decreasing, i.e. $\profile'\leq 0$.
For the solution $W$ constructed in Theorem~\ref{theorem:ASW}, 
it is easy to check that $\partial_\xi W$ satisfies Cauchy problem \eqref{CP:U}.
Consequently we obtain Theorem~\ref{theorem:AS}.
Global existence will be the consequence of the existence of a Lyapunov functional,
 which also allows to deduce the asymptotic stability of traveling waves,
 see also~\cite[Theorem~4]{Achleitner+Hittmeir+Schmeiser:2011}
 for the special case $\theta=2-\alpha$.


\begin{lemma}\label{lem:a_priori}
Suppose that the same assumptions as in Theorem~\ref{theorem:AS} hold.
Let~$W$ be a solution to~\eqref{CP:W} satisfying $W \in C([0,T];H^2)$ 
for some $T > 0$.
Then there exists some positive constant $\delta_1$ 
independent of $T$ such that if $\sup_{0 \le t \le T}\|W(t)\|_{H^2} \le \delta_1$,
the \emph{a-priori} estimate expressed in \eqref{energy_est} holds for $t \in [0,T]$.
\end{lemma}


\begin{proof}
We rewrite the first equation of~\eqref{CP:W},
\[
 \sdiff{W}{t} + (f(\profile+\sdiff{W}{\xi}) -f(\profile) -f'(\profile)\sdiff{W}{\xi}) +(f'(\profile) -\wavespeed) \sdiff{W}{\xi} = \RieszFeller W ,
\]
and test it with $W$,
\begin{multline*}
  \frac12 \sdiff{(W^2)}{t} +\frac12 \sdiff{}{\xi} \{(f'(\profile)-s)\, W^2\} -\frac12 f''(\profile)\profile' W^2 -W \RieszFeller W \\
    = -\big(f(\profile +\partial_\xi W) -f(\profile) -f'(\profile)\sdiff{W}{\xi}\big)\, W .
\end{multline*}
Integrating with respect to $\xi \in \R$, we obtain 
\begin{multline*} 
  \frac12 \sdiff{}{t} \|W\|_{L^2}^2
   - \frac12 \integral{\R}{ f''(\profile)\profile' W^2 }{\xi}
   + \cos \bigPar{ \theta \tfrac{\pi}{2} } \|W\|^2_{\dot{H}^{\alpha/2}} \\
  = - \integral{\R}{ \integrall{0}{1}{ \integrall{0}{\sigma}{ f''(\profile +\gamma \sdiff{W}{\xi}) (\sdiff{W}{\xi})^2 }{\gamma} }{\sigma} W }{\xi} \\
  \leq L(\norm{\partial_\xi W}_{L^\infty}) \norm{W}_{L^\infty} \norm{\partial_\xi W}_{L^2}^2 
\end{multline*}
where $L$ is a positive non-decreasing function.
Due to a Sobolev embedding and the assumption on $W$,
 we deduce $\norm{W(t)}_{W^{1,\infty}} \leq \norm{W(t)}_{H^2}\leq \delta_1$ for all $t\in [0,T]$.
Thus the energy estimate becomes  
\begin{equation} \label{energy:W}
  \frac12 \sdiff{}{t} \|W\|_{L^2}^2
   - \frac12 \integral{\R}{ f''(\profile)\profile' W^2 }{\xi}
   + \cos \bigPar{ \theta \tfrac{\pi}{2} } \|W\|^2_{\dot{H}^{\alpha/2}} \\
  \leq 2\Cone \norm{W}_{L^\infty} \norm{\partial_\xi W}_{L^2}^2 
\end{equation}
for some positive constant $\Cone$ depending on $\delta_1$.
Note that we keep $\norm{W}_{L^\infty}$ for further reference.
Here we used that
 \begin{equation*}
  \integral{\R}{ W \RieszFeller W }{\xi}
    = \integral{\R}{ \RFsymbol(k) |\hat{W}(k)|^2 }{k} 
    = -\cos \bigPar{ \theta \tfrac{\pi}{2} } \|W\|^2_{\dot{H}^{\alpha/2}}
 \end{equation*}
 due to Plancherel's theorem and $\sgn(k)|\hat{W}(k)|^2$ being an odd function.
Similarly, we multiply the first equation of \eqref{CP:U} by $U$, obtaining
\begin{multline*}
\frac12 \partial_t (U^2) 
 + \partial_\xi \Big\{(f(\profile + U) -  f(\profile))U - \integrall{0}{U}{ (f(\profile + \eta) -  f(\profile)) }{\eta} - \frac12 s U^2 \Big\} \\
 +  \profile' \integrall{0}{U}{ (f'(\profile + \eta) -  f'(\profile))  }{\eta} - U \RieszFeller U = 0.
\end{multline*}
Thus, integrating with respect to $\xi \in \R$, we have
\begin{equation} \label{energy:U}
  \frac12 \sdiff{}{t} \|U\|_{L^2}^2
   + \cos \bigPar{ \theta \tfrac{\pi}{2} } \|U\|^2_{\dot{H}^{\alpha/2}} 
   \le \tfrac12 \norm{\profile'}_{L^\infty} L(\norm{U}_{L^\infty})\ \|U\|_{L^2}^2 \, 
   \leq \CUone \norm{U}_{L^2}^2
\end{equation}
 with a positive constant $\CUone$ depending on $\delta_1$.
Next, we differentiate \eqref{CP:U}, obtaining
$\partial_t \partial_\xi U + \partial_\xi^2 \{f(\profile+U)-f(\profile)\} -\wavespeed \partial_\xi^2 U = \RieszFeller \partial_\xi U$.
Testing this equation by $\partial_\xi U$ yields
\begin{multline*}
 \frac12 \partial_t (|\partial_\xi U|^2)
   +\frac12 \partial_\xi \{(f'(\profile + U) -s )(\partial_\xi U)^2\} 
   -\partial_\xi U \RieszFeller  \partial_\xi U  \\
 = -\frac12 \partial_\xi f'(\profile +U)\, (\partial_\xi U)^2
   -\partial_\xi \big((f'(\profile +U) -f'(\profile))\, \profile'\big)\, \partial_\xi U . 
\end{multline*}
Integrating with respect to $\xi \in \R$, we get
\begin{multline*}
 \frac12 \sdiff{}{t} \|\partial_\xi U\|_{L^2}^2
   + \cos \bigPar{ \theta \tfrac{\pi}{2} } \|\partial_\xi U\|^2_{\dot{H}^{\alpha/2}} \\
  =-\frac12 \integral{\R}{ \partial_\xi f'(\profile +U)\, (\partial_\xi U)^2 }{\xi}
   - \integral{\R}{ \partial_\xi \big((f'(\profile +U) -f'(\profile))\, \profile'\big)\, \partial_\xi U }{\xi} ,
\end{multline*}
and hence
\begin{equation} \label{energy:U'}
 \frac12 \sdiff{}{t} \|\partial_\xi U\|_{L^2}^2
   + \cos \bigPar{ \theta \tfrac{\pi}{2} } \|\partial_\xi U\|^2_{\dot{H}^{\alpha/2}} 
   \le \Ctwo \big(\norm{U}_{H^1}^2 + \|\partial_\xi U\|_{L^3}^3\big),
\end{equation}
where $\Ctwo$ is a positive constant depending on $\delta_1$.

By combining \eqref{energy:W}, \eqref{energy:U} and \eqref{energy:U'}, 
we construct the good energy estimate.
For this purpose, we prepare some useful interpolation inequalities.
For $0 \le \sigma \le 2$ and $\varepsilon > 0$, we obtain
\begin{equation}\label{interpol} 
 \|v\|^2_{\dot{H}^1} 
 \leq \varepsilon^{\sigma -2} \|v\|^2_{\dot{H}^{\sigma/2}} + \varepsilon^{\sigma} \|v\|^2_{\dot{H}^{\sigma/2+1}}. 
\end{equation} 
The inequality \eqref{interpol} is proved as follows.
For arbitrary constants $\varepsilon > 0$ and $k \in \R$, we put $h = \varepsilon k$.
Then, by the fact that $h^2 \le |h|^{\sigma} + |h|^{2+\sigma}$ for all $h\in\R$ and $0 \le \sigma \le 2$,
we obtain $k^2 \le \varepsilon^{\sigma -2}|k|^{\sigma} +  \varepsilon^{\sigma}|k|^{2+\sigma}$.  
Thus, by using this inequality and Plancherel's theorem, we arrive at \eqref{interpol}.
On the other hand, for $\sigma > 1/4$, we have
\begin{equation}\label{GN}
 \|v\|_{L^3}^3  \leq C_0 \|v\|_{L^2} \|v\|^2_{H^{\sigma}} \le  2^{\sigma} C_0 \|v\|_{L^2} (\|v\|_{L^2}^2 + \|v\|^2_{\dot{H}^{\sigma}}),   
\end{equation} 
where $C_0$ is a certain positive constant. 
The first interpolation inequality of \eqref{GN} is a generalization 
of the celebrated Gagliardo-Nirenberg inequalities 
(see e.g. \cite{Henry:1981}) to Sobolev spaces with fractional order, 
which was proven by Amann~\cite[Proposition~4.1]{Amann:1985}.
The second inequality holds as a consequence of 
$(1 + |k|^2)^{\sigma} \le 2^{2\sigma}(1 + |k|^{2 \sigma})$
for all $k \in \R$.

We multiply \eqref{energy:U} by $\gamma_1$ and combine the resultant inequality with \eqref{energy:W}, obtaining
\begin{multline*}
 \frac12 \sdiff{}{t} (\|W\|_{L^2}^2 + \gamma_1 \|U\|_{L^2}^2) 
   - \frac12 \integral{\R}{ f''(\profile) \profile' W^2 }{\xi} 
   + \cos \bigPar{ \theta \tfrac{\pi}{2} } (\|W\|^2_{\dot{H}^{\alpha/2}} + \gamma_1 \|U\|^2_{\dot{H}^{\alpha/2}} ) \\
 \le \gamma_1 \CUone \|U\|_{L^2}^2 + 2\Cone \norm{W}_{L^\infty} \|\partial_\xi W\|^2_{L^{2}}, 
\end{multline*}
where $\gamma_1$ is a positive constant to be determined later.
By the fact that $\partial_\xi W = U$, we can apply \eqref{interpol} with $v = W$ and $\sigma = \alpha$
to the above inequality, and get  
\begin{multline*}
 \frac12 \sdiff{}{t} (\|W\|_{L^2}^2 + \gamma_1 \|U\|_{L^2}^2) 
   - \frac12 \integral{\R}{ f''(\profile) \profile' W^2 }{\xi} \\
   + \{ \cos \bigPar{ \theta \tfrac{\pi}{2} } - \varepsilon_1^{\alpha - 2}\gamma_1 \CUone \} \|W\|^2_{\dot{H}^{\alpha/2}} 
  + \gamma_1 \{\cos \bigPar{ \theta \tfrac{\pi}{2} } - \varepsilon_1^{\alpha}\CUone\} \|U\|^2_{\dot{H}^{\alpha/2}} \\
   \le 2\Cone \norm{W}_{L^\infty} \|\partial_\xi W\|^2_{L^{2}}. 
\end{multline*}
Therefore, we choose $\varepsilon_1$ satisfying $4 \varepsilon_1^\alpha \CUone = \cos(\theta \pi/2)$, and $\gamma_1 = \varepsilon_1^2$ to get
\begin{multline} \label{energy:WU}
 \frac12 \sdiff{}{t} (\|W\|_{L^2}^2 + \gamma_1 \|U\|_{L^2}^2)
  - \frac12 \integral{\R}{ f''(\profile) \profile' W^2 }{\xi} \\
 + \frac34 \cos \bigPar{ \theta \tfrac{\pi}{2} } (\|W\|^2_{\dot{H}^{\alpha/2}} + \gamma_1 \|U\|^2_{\dot{H}^{\alpha/2}} )
  \le 2\Cone \norm{W}_{L^\infty} \|\partial_\xi W\|^2_{L^{2}}. 
\end{multline}
Similarly we multiply \eqref{energy:U'} by $\gamma_2$ and combine the resultant inequality with~\eqref{energy:WU}.
Furthermore, applying \eqref{interpol} to the resultant inequality, we have 
\begin{multline*}
 \frac12 \sdiff{}{t} (\|W\|_{L^2}^2 + \gamma_1 \|U\|_{L^2}^2 +  \gamma_2 \|\partial_\xi U\|_{L^2}^2) 
   - \frac12 \integral{\R}{ f''(\profile) \profile' W^2 }{\xi} 
   + \{ \frac34 \cos \bigPar{ \theta \tfrac{\pi}{2} } - \varepsilon_2^{\alpha - 2}\gamma_2 \Ctwo \} \|W\|^2_{\dot{H}^{\alpha/2}} \\
   + \{ \frac34 \gamma_1 \cos \bigPar{ \theta \tfrac{\pi}{2} } - (1 + \varepsilon_2^{-2}) \varepsilon_2^{\alpha}\gamma_2 \Ctwo\} \|U\|^2_{\dot{H}^{\alpha/2}} 
   + \gamma_2 \{\cos \bigPar{ \theta \tfrac{\pi}{2} } - \varepsilon_2^{\alpha}\Ctwo\} \|\partial_\xi U\|^2_{\dot{H}^{\alpha/2}} \\
 \le 2\Cone \norm{W}_{L^\infty} \|\partial_\xi W\|^2_{L^{2}} + \gamma_2 \Ctwo \|\partial_\xi U\|_{L^3}^3.
\end{multline*}
Then, choosing $\varepsilon_2$ such that $4 \varepsilon_2^\alpha \Ctwo = \cos(\theta \pi/2)$,
and $\gamma_2 = \min\{ \varepsilon_2^2, \gamma_1(1 + \varepsilon_2^{-2})^{-1}\}$, yields
\begin{multline}\label{energy:WUU'}
 \frac12 \sdiff{}{t} (\|W\|_{L^2}^2 + \gamma_1 \|U\|_{L^2}^2 +  \gamma_2 \|\partial_\xi U\|_{L^2}^2) 
    - \frac12 \integral{\R}{ f''(\profile) \profile' W^2 }{\xi} \\
 + \frac12 \cos \bigPar{ \theta \tfrac{\pi}{2} } (\|W\|^2_{\dot{H}^{\alpha/2}} 
    + \gamma_1 \|U\|^2_{\dot{H}^{\alpha/2}} + \gamma_2 \|\partial_\xi U\|^2_{\dot{H}^{\alpha/2}}) \\
 \le 2\Cone \norm{W}_{L^\infty} \|\partial_\xi W\|^2_{L^{2}} + \gamma_2 \Ctwo \|\partial_\xi U\|_{L^3}^3.
\end{multline}
We introduce the energy and dissipation norms as follows.
\begin{align*}
 E(t)^2 &:=  \sup_{0 \le \tau \le t}(\|W(\tau)\|_{L^2}^2 + \gamma_1 \|U(\tau)\|_{L^2}^2 +  \gamma_2 \|\partial_\xi U(\tau)\|_{L^2}^2), \\
 D(t)^2 &:=  \integrall{0}{t}{(\|W(\tau)\|^2_{\dot{H}^{\alpha/2}} 
  + \gamma_1 \|U(\tau)\|^2_{\dot{H}^{\alpha/2}} + \gamma_2 \|\partial_\xi U(\tau)\|^2_{\dot{H}^{\alpha/2}}) }{\tau}.
\end{align*}
Then, integrating \eqref{energy:WUU'} with respect to $t$, we have
\begin{multline*}
 \|W\|_{L^2}^2 + \gamma_1 \|U\|_{L^2}^2 +  \gamma_2 \|\partial_\xi U\|_{L^2}^2 +  \cos \bigPar{ \theta \tfrac{\pi}{2} } D(t)^2
    - \integrall{0}{t}{\integral{\R}{ f''(\profile) \profile' W^2 }{\xi} }{\tau} \\
 \le E_0^2 + \integrall{0}{t}{ \big(4\Cone \norm{W}_{L^\infty} \|U\|^2_{L^2} + 2 \gamma_2 \Ctwo \|\partial_\xi U\|_{L^3}^3\big) }{\tau},
\end{multline*}
where we define $E_0^2 := \|W_0\|_{L^2}^2 + \gamma_1 \|U_0\|_{L^2}^2 +  \gamma_2 \|\partial_\xi U_0\|_{L^2}^2$.
Thus, by employing \eqref{interpol}, and \eqref{GN} with $v = \partial_\xi U$ and $\sigma = \alpha/2$, 
 we arrive at
\begin{equation*}
 E(t)^2  + \cos \bigPar{ \theta \tfrac{\pi}{2} } D(t)^2 - \integrall{0}{t}{ \integral{\R}{ f''(\profile) \profile' W^2 }{\xi} }{\tau} 
   \le E_0^2 + C E(t)D(t)^2
\end{equation*}
for some positive constant~$C$.
Finally, using the fact that $E(t) \le \delta_1^2 C$, we arrive at the desired a-priori estimate.
\end{proof}

\bigskip

\begin{proof}[Proof of Theorem~\ref{theorem:ASW}]
The existence of global-in-time solutions to the initial value problem \eqref{CP:W} 
can be obtained by the continuation argument based on a local existence result in Proposition~\ref{prop:CP:H2} 
combined with the {\it a-priori} estimate in Lemma~\ref{lem:a_priori}.
Because the argument is standard, we may omit the details here.
In the rest of this proof, we prove only the asymptotic stability result \eqref{asyW}.

To this end, we prepare the following interpolation inequality.
For $0 \le \sigma \le 2$, we have
\begin{equation*}
 \|v\|_{\dot{H}^\sigma} 
 \leq 2(\|v\|_{\dot{H}^{\sigma/2}} + \|v\|_{\dot{H}^{\sigma/2+1}}), 
\end{equation*} 
by using the fact that $k^{2\sigma} \le 2(|k|^{\sigma} + |k|^{2+\sigma})$.  
By virtue of this interpolation inequality, \eqref{interpol}, and the first equation of \eqref{CP:U}, 
we have 
 \begin{equation*}
 	\begin{split}
 \|\partial_t U\|_{L^2} 
 & \le \|\RieszFeller U\|_{L^2} + \|\{f'(\profile+U)-f'(\profile)\}\profile'\|_{L^2} + \| \{f'(\profile + U) -\wavespeed\}\partial_\xi U\|_{L^2} \\
 & \le \|U\|_{\dot{H}^\alpha} + C \|U\|_{H^1} 
 \le C \sum_{\ell = 0}^2 \|W\|_{\dot{H}^{\alpha/2 + \ell}}.
	\end{split}
 \end{equation*}
Thus, by the above estimate, we compute that
  \begin{equation*}
 	\begin{split}
\Big| \sdiff{}{t} \| U \|^2_{L^2} \Big| \le \|U\|_{L^2}^2 +  \|\partial_t U\|_{L^2}^2 
 \le C \sum_{\ell = 0}^2 \|W\|_{\dot{H}^{\alpha/2 + \ell}}^2.
	\end{split}
 \end{equation*}
This estimate and \eqref{interpol} with \eqref{energy_est} tell us that 
$\|U(\cdot)\|_{L^2}^2 \in W^{1,1}(0,\infty)$, and hence $\|U(t)\|_{L^2} \to 0$ as $t \to \infty$.
Finally, employing the Sobolev inequality that
$\| v \|_{L^\infty} \le \sqrt{2} \|v\|_{L^2}^{1/2} \|\partial_\xi v\|_{L^2}^{1/2}$,
we arrive at the desired result.
\end{proof}


\section{Convergence rate toward traveling waves} \label{sec:Rates}

We consider the convergence rate of the solution toward the corresponding traveling waves.
Kawashima, Nishibata and Nishikawa~\cite{Kawashima+etal:2004} proposed an $L^p$ energy method 
to study the asymptotic stability and the associated convergence rates of planar viscous rarefaction waves 
of multi-dimensional viscous conservation laws.
When the authors obtain the convergence estimate, they derived the $L^1$ estimate 
by using the energy method associated with the sign function.
This approach is useful. It is however difficult to apply this method because of a Riesz-Feller operator.
To overcome this difficulty, we employ not only the energy method but also 
the representation of the mild solution. 
Precisely, our purpose in this section is to derive the following theorem.

\begin{theorem}\label{theorem:CRW}
Suppose that the same assumptions as in Theorem~\ref{theorem:AS} and $f\in C^\infty(\R)$ hold.
Then the Cauchy problem \eqref{CP:W} with $W_0\in W^{1,1}(\R) \cap W^{1,\infty}(\R)$ 
has a unique global solution $W(\xi,t)$ satisfying 
\[ W \in C([0,\infty);W^{1,1}(\R)\cap H^1(\R))\cap  L^\infty(0,\infty;W^{1,\infty}(\R)) \]
with estimates \eqref{low_energy_est} and~\eqref{est-L1}.
Moreover, there exists a positive constant~$\delta_1$ such that if $\norm{W_0}_{W^{1,1}}\leq \delta_1$ then
\begin{equation}\label{decay_est_W}
 \|W(t)\|_{H^1} \le  C E_1\ (1+t)^{-1/(2\alpha)}
\end{equation}
for $t \ge 0$, where $E_1 := \|W_0\|_{H^1} + \|W_0\|_{W^{1,1}}$ and $C$ is a certain positive constant independent of $t$.
\end{theorem}

The proof of the existence of global-in-time solutions is based on results for the Cauchy problem~\eqref{eq:FCL} with fractional Laplacian~\cite{Droniou+etal:2003}
 and its extension to the Cauchy problem~\eqref{eq:FCL} with Riesz-Feller operators~\cite{Achleitner+etal:2012}.
There the assumption $f\in C^\infty(\R)$ is made to simplify the presentation.
The method is applicable also in case of $f\in C^k(\R)$, $k\geq 2$,
 but yields a lower regularity for the unique solution $u$. 
 
\begin{lemma} \label{lem:CP:W11nW1infinity}
Suppose that $f\in C^\infty(\R)$ 
 and $W_0\in W^{1,1}(\R)\cap W^{1,\infty}(\R)$.
Then Cauchy problem~\eqref{CP:W} has a unique mild solution $W \in C([0,T];W^{1,1}(\R)\cap H^1(\R))\cap L^\infty(0,T;W^{1,\infty}(\R))$ for any $T>0$ with
\begin{align}
 \norm{W(t)}_{L^1} &\leq \norm{W_0}_{L^1} + L(\sup_{\tau\in[0,t]} \norm{\sdiff{W}{\xi}(\tau)}_{L^\infty}) \norm{\sdiff{W_0}{\xi}}_{L^1}\ t \ , \label{est:W:L1} \\
 \norm{\sdiff{W}{\xi}(t)}_{L^1} &\leq \norm{\sdiff{W_0}{\xi}}_{L^1} \ , \label{est:U:L1} \\ 
 \norm{W(t)}_{L^\infty} &\leq \norm{\sdiff{W_0}{\xi}}_{L^1} \ ,  \label{est:W:Linfinity} \\
 \norm{\sdiff{W}{\xi}(t)}_{L^\infty} &\leq \norm{\sdiff{W_0}{\xi}}_{L^\infty} +2\norm{\profile}_{L^\infty} \ ,  \label{est:U:Linfinity}
\end{align}
for $0 \le t \le T$, where  $L$ is a positive non-decreasing function.
Moreover, for any positive time $t_0>0$,
 $W\in C_b^\infty(\R\times (t_0,\infty))$
 and it is a classical solution of the first equation of \eqref{CP:W}.
\end{lemma}
\begin{proof}
We use again $U =\sdiff{W}{\xi}$
 and analyze the Cauchy problem~\eqref{CP:U} with initial datum $U_0 :=\sdiff{W_0}{\xi} \in L^1(\R)\cap L^{\infty}(\R)$ first.
We recall $U = u -\profile$
 where $u$ and $\profile$ solve equation~\eqref{CLND+MCF}, 
 and $\profile$ is a monotone decreasing function satisfying $\lim_{\xi\to\pm\infty} \profile(\xi) =\upm$.
Thus, $u_0 := U_0 +\profile$ is essentially bounded.
Due to~\cite[Theorem~1]{Droniou+etal:2003} and its extension to equations with Riesz-Feller operators in~\cite{Achleitner+etal:2012},
 the Cauchy problem for~\eqref{CLND+MCF} with initial datum $u_0\in L^\infty(\R)$ has a (unique) solution 
 which satisfies $\norm{u(t)}_{L^\infty(\R)}\leq \norm{u_0}_{L^\infty(\R)}$ for all $t\geq 0$;
in fact, the solution $u$ takes values between the essential lower and upper bounds of $u_0$.
Therefore, $U(t) = u(t) -\profile \in L^\infty(\R_\xi)$ for all $t\geq 0$
 and estimate~\eqref{est:U:Linfinity} follows.

Due to~\cite[Remark 1.2]{Droniou+etal:2003} and its extension to equations with Riesz-Feller operators,
 equation~\eqref{CLND+MCF} supports an $L^1$ contraction principle:
If $u_0$, $v_0 \in L^\infty(\R)$ satisfy $u_0 -v_0 \in L^1(\R)$,
 then the associated solutions $u$ and $v$ of the Cauchy problem for~\eqref{CLND+MCF} satisfy
 $\norm{u(t) -v(t)}_{L^1(\R)} \leq \norm{u_0 -v_0}_{L^1(\R)}$ for all $t\geq 0$. 
Therefore, $U(t) = u(t) -\profile \in L^1(\R_\xi)$ with $\norm{U(t)}_{L^1}\leq \norm{u_0 -\profile}_{L^1} =\norm{U_0}_{L^1}$ for all $t\geq 0$,
 which 
 implies estimate~\eqref{est:U:L1}.
Moreover, its primitive $W(t) \in L^\infty(\R_\xi)$ for all $t\geq 0$, since
 \[ \norm{W(t)}_{L^\infty}
      = \Norm{ \integrall{-\infty}{\xi}{ \sdiff{W}{y}(y,t) }{y} }_{L^\infty} 
      \leq \integrall{-\infty}{\infty}{ \abs{\sdiff{W}{y}(y,t)} }{y}
			= \norm{\sdiff{W}{\xi}(t)}_{L^1} \ .
 \]
Then, we are left to prove that $W(t) \in L^1(\R_\xi)$ for all $t\geq 0$ 
 and the stated continuity in time.
Considering the mild formulation~\eqref{mildform:W}, 
 we obtain the estimate 
 \begin{align}
  \|W(t)\|_{L^1} 
  &\leq \|\Green(t)\ast W_0 \|_{L^1} 
  + \integrall{0}{t}{ \| \Green(t-\tau) \ast \{f(\profile+U)-f(\profile)-\wavespeed U\} \|_{L^1} }{\tau} \nonumber \\
  &\leq \|W_0 \|_{L^1} + \integrall{0}{t}{ \|f(\profile+U)-f(\profile)-\wavespeed U\|_{L^1} }{\tau} \nonumber \\
  &\leq \|W_0\|_{L^1} + \integrall{0}{t}{ \big(\widetilde{L}(\norm{U(\tau)}_{L^\infty})\ \|U(\tau)\|_{L^1}\big) }{\tau} \nonumber \\
  &\leq \|W_0\|_{L^1} + \widetilde{L}(\norm{\sdiff{W_0}{\xi}}_{L^\infty} +2\norm{\profile}_{L^\infty})\ \|U_0\|_{L^1}\ t \ , \label{low_L1_W}
 \end{align}
 for $t \ge 0$ by using the local Lipschitz continuity of $f$
 and the previous estimates on $U =\sdiff{W}{\xi}$;
 again, $\widetilde{L}$ is a positive non-decreasing function.
Moreover, for any positive time $t_0>0$, $U \in C_b^\infty(\R\times (t_0,\infty))$
 and $U = \partial_\xi W$ satisfies the first equation of \eqref{CP:U} in the classical sense,
 see~\cite{Droniou+etal:2003,Achleitner+Hittmeir+Schmeiser:2011}.
Due to integrability of $U$, also $W$ is a global-in-time solution of \eqref{CP:W},
 and $W\in C_b^\infty(\R\times (t_0,\infty))$ is a classical solution of the first equation of \eqref{CP:W} for all $t\geq t_0>0$.

To prove that $W \in C([0,T];W^{1,1}(\R)\cap H^1(\R))$,
 we will use the mild formulation 
\begin{equation} \label{mildform:W:2}
 W(t) = \Green(t)\ast W_0 
  - \integrall{0}{t}{ \Green(t-\tau)\ast F(\profile,\sdiff{W}{\xi}) }{\tau},
\end{equation} 
where $F(\profile,\sdiff{W}{\xi}) := f(\profile+\sdiff{W}{\xi})-f(\profile)-\wavespeed \sdiff{W}{\xi}$.
The first summand on the right hand side satisfies 
 $\Green(\cdot)\ast W_0 \in C([0,T];W^{1,1}(\R)\cap H^1(\R))$,
 due to the assumptions on $W_0$ and the strong continuity of the semigroup in Lemma~\ref{lem:SFDE:semigroup}.
To prove continuity of the second summand, 
 \[ \mathcal{G}_2[W](t) :=\integrall{0}{t}{ \Green(t-\tau)\ast F(\profile,\sdiff{W}{\xi}) }{\tau} \ , \] 
 we use the estimates~\eqref{est:W:L1}--\eqref{est:U:Linfinity}
 and the strong continuity of the semigroup in Lemma~\ref{lem:SFDE:semigroup}.
In particular, we assume w.l.o.g. $0<t_1<t_2$ and rewrite
\begin{align*}
 &\mathcal{G}_2[W](t_1) -\mathcal{G}_2[W](t_2) \\
   &\ = \integrall{0}{t_1}{ (\Green(t_1-\tau) -\Green(t_2-\tau))\ast F(\profile,\sdiff{W}{\xi}) }{\tau}
	    +\integrall{t_1}{t_2}{ \Green(t_2-\tau)\ast F(\profile,\sdiff{W}{\xi}) }{\tau} \\
   &\ = \integrall{0}{t_1}{ \big[\Green(t_1-\tau)\ast F(\profile,\sdiff{W}{\xi}) 
	                         -\Green(t_2-t_1) \ast\big(\Green(t_1-\tau)\ast F(\profile,\sdiff{W}{\xi})\big) \big] }{\tau} \\
	 &\ \quad +\integrall{t_1}{t_2}{ \Green(t_2-\tau)\ast F(\profile,\sdiff{W}{\xi}) }{\tau}
\end{align*} 
using the semigroup property~\ref{K:SFDE:prop3}.
The first summand converges to zero as $t_2\to t_1$ in the $W^{1,p}$-norms, $p=1,2$,
 due to the Dominated Convergence Theorem, the strong continuity of the semigroup in Lemma~\ref{lem:SFDE:semigroup}
 and that $\integrall{0}{t_1}{ \big(\Green(t_1-\tau) \ast F(\profile,\sdiff{W}{\xi})\big) }{\tau} \in W^{1,1}(\R)\cap H^1(\R)$.
Similarly, the second summand converges to zero as $t_2\to t_1$ in the $W^{1,p}$-norms, $p=1,2$, 
 since $\Green(t_2-\cdot) \ast F(\profile,\sdiff{W}{\xi}) \in L^1((t_2,t_1);W^{1,1}(\R)\cap W^{1,\infty}(\R))$.
Thus, the right hand side of~\eqref{mildform:W:2} is continuous in time with respect to the $W^{1,p}$-norms, $p=1,2$,
 hence $W \in C([0,T];W^{1,1}(\R)\cap H^1(\R))$.
Finally, $W \in L^\infty(0,T;W^{1,\infty}(\R))$ follows from the estimates~\eqref{est:W:Linfinity}-\eqref{est:U:Linfinity}.
\end{proof}

Next we prove the following {\it a-priori} estimate obtained by Lemma~\ref{lem:CP:W11nW1infinity}.
%
%
\begin{lemma}\label{lem:a_priori_CR}
Suppose that the same assumptions as in Theorem~\ref{theorem:CRW} hold.
Let $W(\xi,t)$ be a solution to \eqref{CP:W} satisfying $W \in C([0,T]; W^{1,1}(\R)\cap H^1(\R))\cap L^\infty(0,T;W^{1,\infty}(\R))$ 
for any $T > 0$.
Then there exists some positive constants $\delta_1$ 
independent of $T$ such that if $\norm{W_0}_{W^{1,1}} \le \delta_1$,
the {\it a-priori} estimates
\begin{multline}\label{low_energy_est}
\|W(t)\|_{H^1}^2 + C \integrall{0}{t}{(\|W(\tau)\|^2_{\dot{H}^{\alpha/2}} + \|W(\tau)\|^2_{\dot{H}^{\alpha/2 + 1}}) }{\tau} 
    - \integrall{0}{t}{ \integral{\R}{ f''(\profile) \profile' W^2 }{\xi} }{\tau}
      \le \|W_0\|_{H^1}^2 \ ,
\end{multline}
 \begin{equation}\label{est-L1}
  \|W(t)\|_{W^{1,1}} \leq C (\|W_0\|_{W^{1,1}} + \|W_0\|_{H^1}^2) \ ,
 \end{equation}
hold for $t \in [0,T]$, where $C$ is a constant independent of time $t$.
\end{lemma}

\begin{proof}
Following the proof of Lemma~\ref{lem:a_priori},
 we deduce again estimate~\eqref{energy:WU}, i.e.
\begin{multline*} 
 \frac12 \sdiff{}{t} (\|W\|_{L^2}^2 + \gamma_1 \|U\|_{L^2}^2) 
  - \frac12 \integral{\R}{ f''(\profile) \profile' W^2 }{\xi} \\
 + \frac34 \cos \bigPar{ \theta \tfrac{\pi}{2} } (\|W\|^2_{\dot{H}^{\alpha/2}} + \gamma_1 \|U\|^2_{\dot{H}^{\alpha/2}} )
  \le L(\norm{\sdiff{W}{\xi}}_{L^\infty})\ \norm{W}_{L^\infty} \|\partial_\xi W\|^2_{L^{2}}
\end{multline*}
 for some positive non-decreasing function~$L$.
Integrating this inequality with respect to time 
 and using \eqref{interpol}, the estimates~\eqref{est:W:Linfinity}--\eqref{est:U:Linfinity} as well as the smallness of $\norm{W_0}_{W^{1,1}}$,
 we arrive at \eqref{low_energy_est}.

Thus it remains to prove \eqref{est-L1}.
Due to Lemma~\ref{lem:CP:W11nW1infinity}, for all $t_0>0$,
 $W\in C_b^\infty(\R\times (t_0,\infty))$ and it is a classical solution of the first equation of \eqref{CP:W}.
Therefore we can adapt the $L^1$ energy method introduced by 
Kawashima, Nishibata and Nishikawa~\cite{Kawashima+etal:2004}. 
For a non-negative function $\rho: \R\to\R$ satisfying $\rho\in\C_0^\infty(\R)$ and $\integral{\R}{\rho(x)}{x}=1$, 
  the convolution operator $\rd\ast$ with $\rd(x) = \delta^{-1} \rho(x/\delta)$ is a Friedrichs' mollifier. 
 We introduce the functions
  \[ \sd(x) := (\rd \ast \sgn)(x)  \XX{and} \Sd(x) := \integrall{0}{x}{ \sd(\xi) }{\xi} \,, \]
  in which the signature function~$\sgn(x)$ is defined by
  \[
    \sgn(x) := 
    \begin{cases}
    -1 & \text{for } x<0 \,, \\
      0 & \text{for } x=0 \,, \\
      1 & \text{for } x>0 \,.
    \end{cases}
  \]
 Note that the convergence of $\sd(x)\to\sgn(x)$ as $\delta\to 0$ is in the sense of 
  a weak~$\star$ convergence in $L^\infty(\R)$, respectively,
	a strong convergence in $L_{loc}^q(\R)$, $1\leq q < \infty$. 
 The function $\sd(x)$ satisfies $\sdd(x)=2\rd(x)\geq 0$ 
  and $\sd(0)=0$ by choosing $\rho$ to be an even function. 
 Moreover $\Sd(x)\to |x|$ converges strongly in $L^1(\R)$ as $\delta\to 0$. 
 
To estimate $\norm{W(t)}_{W^{1,1}}$,
 we recall that $\norm{U(t)}_{L^1} \leq \norm{U_0}_{L^1}$ for all $t\in [0,T]$,
 due to estimate~\eqref{est:U:L1} in Lemma~\ref{lem:CP:W11nW1infinity}.
Next we show that
\begin{equation}\label{est-L1-W}
\|W(t)\|_{L^1} \leq C \|W_0\|_{W^{1,1}} + C\|W_0\|_{H^1}^2
\end{equation}
for $t\in [0,T]$.
We will use estimate~\eqref{est:W:L1} for small times $t\leq 1$,
 and derive~\eqref{est-L1-W} for large times $t\geq 1$:
We multiply the first equation of \eqref{CP:W} by $\sd(W)=(\rd\ast\sgn)(W)$ and obtain
  \begin{equation} \label{eq:U:primitive:L1}
    \partial_t \Sd(W) + \sd(W) \{h(\profile+U)-h(\profile)\} = \sd(W) \RieszFeller W \ ,
  \end{equation}
 where $h(v):=f(v)-\wavespeed v$ is a convex function.
 We integrate equation~\eqref{eq:U:primitive:L1} 
 over $\R \times [t_0,t]$ and derive
  \begin{multline}\label{eq:U:primitive:L1:approximate}
    \integrall{t_0}{t}{ \integral{\R}{  \partial_t \Sd(W) }{x}  }{\tau} 
    + \integrall{t_0}{t}{ \integral{\R}{  \sd(W) \{h(\profile+U)-h(\profile)\} }{x} }{t} 
    = \integrall{t_0}{t}{ \integral{\R}{  \sd(W) \RieszFeller W }{x} }{t}.
  \end{multline}
The first integral satisfies, due to Fubini's theorem and the strong convergence of $\Sd$ in $L^1$, 
\begin{multline}\label{St-est} 
    \integrall{t_0}{t}{ \integral{\R}{  \partial_t \Sd(W) }{x} }{\tau}
    = \integral{\R}{ \{ \Sd(W(x,t))-\Sd(W(x,t_0)) \} }{x} 
    \to \|W(t)\|_{L^1} - \|W(t_0)\|_{L^1} 
  \end{multline}
 as $\delta \to 0$.
Next, we prove that the integral on the right-hand side of~\eqref{eq:U:primitive:L1:approximate} is non-positive,
 \begin{equation}\label{DW-est}
 \integrall{t_0}{t}{ \integral{\R}{ \sd(W) \RieszFeller[W] }{x} }{\tau}  \leq 0.
 \end{equation}
 Indeed, $\Sd\in C^2(\R)$ is a convex function with $\Sd'=\sd$ and $\Sd''=\sd'=2\rd\geq 0$.
 Moreover, under our assumptions,
  $W(\cdot,t)\in H^1(\R)$ for $t\geq 0$ and $W\in C_b^\infty(\R\times (t_0,\infty))$ for $t_0 >0$.  
 Thus, $\lim_{\xi\to\pm\infty} W(\xi,t) = 0$ and $S_\delta(W) \in C^2_b$ with  
  \[ s_\delta(W)\, \RieszFeller[W] = S_\delta'(W)\, \RieszFeller[W] \leq \RieszFeller[S_\delta(W)] \,, \]
  due to Lemma~\ref{lemma:CI}.
 Consequently, 
  \[ \integral{\R}{ s_\delta(W)\, \RieszFeller[W] }{x} \leq \integral{\R}{ \RieszFeller[S_\delta(W)] }{x} = 0\,, \]
 due to Proposition~\ref{prop:RieszFeller:estimate}.
We estimate the second term on the left-hand side of~\eqref{eq:U:primitive:L1:approximate} as follows.
Using the fact that $|s_\delta (W)| \le 1$ and $h(\profile+U)-h(\profile) = h'(\profile)U + O(|U|^2)$, we have
 \begin{equation*}
   \integral{\R}{ \sd(W) \{h(\profile+U)-h(\profile)\} }{\xi} 
      = \integral{\R}{  \sd(W) h'(\profile)U }{\xi} + R
  \end{equation*}
  with $|R| \le L(\norm{U}_{L^\infty})\ \|U\|_{L^2}^2 /2$.
  Furthermore, we compute from the fact $U = \partial_\xi W$ that
   \begin{equation*}
   \integral{\R}{  \sd(W) h'(\profile)U }{\xi} 
      = - \integral{\R}{  \Sd(W) h''(\profile)  \profile' }{\xi} \geq 0,
  \end{equation*}
 since 
 the function $\Sd$ is non-negative with $\Sd(0)=0$,
 $h\in C^2(\R)$ is a convex function,
 and $\profile$ is a monotone decreasing traveling wave profile.
 Therefore, employing the previous estimates and taking the limit $\delta\to 0$ 
 in equation~\eqref{eq:U:primitive:L1:approximate} yields
 \begin{multline}\label{high_L1_W}
 \|W(t)\|_{L^1} 
  \leq \|W(t_0)\|_{L^1} + L(\norm{\sdiff{W_0}{\xi}}_{L^\infty} +2\norm{\profile}_{L^\infty})\ \integrall{t_0}{t}{ \|U(\tau)\|_{L^2}^2 }{\tau} \\
  \leq \|W(t_0)\|_{L^1} + C\|W_0\|_{H^1}^2 
 \end{multline}
 for $t \ge t_0 >0$ and some positive constant~$C$; 
 here we used~\eqref{low_energy_est} and~\eqref{interpol}.
 The estimate \eqref{high_L1_W} is valid for an arbitrary positive constant $t_0$.
 Thus we can estimate from~\eqref{high_L1_W} and~\eqref{low_L1_W} that
 \begin{equation*}
 \|W(t)\|_{L^1}  \leq \|W(1)\|_{L^1} +  C\|W_0\|_{H^1}^2 
  \leq \|W_0\|_{L^1} + C\|U_0\|_{L^1} +  C\|W_0\|_{H^1}^2 
 \end{equation*}
for $t \ge 1$.
Eventually, combining this estimate and \eqref{low_L1_W} again, 
we arrive at the desired estimate \eqref{est-L1-W}.
\end{proof}

\begin{proof}[Proof of Theorem~\ref{theorem:CRW}]
The existence of the global solution follows from Lemma~\ref{lem:CP:W11nW1infinity}
 and the a-priori estimates in Lemma~\ref{lem:a_priori_CR}.
We derive just the decay estimate \eqref{decay_est_W}.
To this end, we first introduce the following Nash inequality:
 \begin{equation} \label{ineq:Nash}
  \|v\|_{L^2}^{2(1+2\sigma)} \leq C_\sigma \|v\|_{L^1}^{4\sigma} \|v\|_{\dot{H}^\sigma}^2 
 \end{equation}
for $\sigma > 0$ and $v \in L^1(\R) \cap H^\sigma (\R)$,
where $C_\sigma$ is a positive constant which depends on $\sigma$.
Following the proof of Lemma~\ref{lem:a_priori},
 we deduce again estimate~\eqref{energy:WU}.
Multiplying this inequality with $(1+\tau)^\beta$ for $\beta \in \mathbb{R}$ and integrating over $\tau\in [0,t]$, we obtain
\begin{align*}
 \mathcal{E}_{\beta}(t)^2
  & -\integrall{0}{t}{ (1+\tau)^\beta \integral{\R}{ f''(\profile) \profile' W^2 }{\xi} }{\tau}
    + \frac32 \cos \bigPar{ \theta \tfrac{\pi}{2} } \integrall{0}{t}{ \mathcal{D}_\beta(\tau)^2 }{\tau}  \\
  &\le \|W_0\|_{L^2}^2 + \gamma_1 \|U_0\|_{L^2}^2
    + \beta \integrall{0}{t}{ \mathcal{E}_{\beta-1}(\tau)^2  }{\tau} \\
  & \qquad + L(\norm{\sdiff{W_0}{\xi}}_{L^\infty} +2\norm{\profile}_{L^\infty})\ \integrall{0}{t}{ (1+\tau)^\beta \|W\|_{L^\infty} \|\partial_\xi W\|^2_{L^{2}}}{\tau}
\end{align*}
where $\mathcal{E}_\beta(t)^2 := (1+t)^\beta (\|W(t)\|_{L^2}^2 + \gamma_1 \|U(t)\|_{L^2}^2)$, and
\begin{equation*}
 \mathcal{D}_\beta(t)^2 :=  (1+t)^\beta (\|W(t)\|^2_{\dot{H}^{\alpha/2}} + \gamma_1 \|U(t)\|^2_{\dot{H}^{\alpha/2}} ).
\end{equation*}
We compute via Nash's inequality~\eqref{ineq:Nash} with $\sigma=\alpha/2$ and Young's inequality that
\begin{align*}
 (1+t)^{\beta-1} \|v\|_{L^2}^2  
  & \leq C (1+t)^{\beta-1} \|v\|_{\dot{H}^{\alpha/2}}^{\frac{2}{1+\alpha}} \|v\|_{L^1}^{\frac{2\alpha}{1+\alpha}} \\
  & = C \{ (1+t)^{\beta} \norm{v}_{\dot{H}^{\alpha/2}}^2 \}^{\frac{1}{1+\alpha}}
  \{ (1+t)^{\beta-\frac{1+\alpha}{\alpha}} \|v\|_{L^1}^2 \}^{\frac{\alpha}{1+\alpha}}\\
   & \le \epsilon (1+t)^{\beta} \norm{v}_{\dot{H}^{\alpha/2}}^2 
  + C_\epsilon (1+t)^{\beta-\frac{1+\alpha}{\alpha}} \|v\|_{L^1}^2 \ ,
\end{align*}
for all $\epsilon>0$ and some positive constant $C_\epsilon$.
Thus we get 
$\mathcal{E}_{\beta-1}(t)^2 
\le \epsilon \mathcal{D}_{\beta}(t)^2 + C_\epsilon (1+t)^{\beta-\frac{1+\alpha}{\alpha}} (\|W\|_{L^1}^2 + \gamma_1 \|U\|_{L^1}^2) 
$.
Therefore, employing this estimate and \eqref{est-L1}, we obtain
\begin{align*}
 \mathcal{E}_{\beta}(t)^2 
 & -\integrall{0}{t}{ (1+\tau)^\beta \integral{\R}{ f''(\profile) \profile' W^2 }{\xi} }{\tau}
   + \Big\{ \frac32 \cos \bigPar{ \theta \tfrac{\pi}{2} } - \epsilon \beta \Big\}\integrall{0}{t}{ \mathcal{D}_\beta(\tau)^2 }{\tau}  \\[-2mm]
 & \le  \|W_0\|_{L^2}^2 + \gamma_1 \|U_0\|_{L^2}^2
   + \beta C_\epsilon \integrall{0}{t}{ \ (1+\tau)^{\beta-\frac{1+\alpha}{\alpha}} (\|W\|_{L^1}^2 + \gamma_1 \|U\|_{L^1}^2)  }{\tau} \\[-2mm]
 &\qquad + L(\norm{\sdiff{W_0}{\xi}}_{L^\infty} +2\norm{\profile}_{L^\infty})\ \integrall{0}{t}{ (1+\tau)^\beta \|W\|_{L^\infty} \|\partial_\xi W\|^2_{L^{2}}}{\tau} \\[-2mm]
&   \le  C\|W_0\|_{H^1}^2 
 + \beta C_\epsilon (\|W_0\|_{H^1}^2 + \|W_0\|_{W^{1,1}})^2 \integrall{0}{t}{ \ (1+\tau)^{\beta-\frac{1+\alpha}{\alpha}}   }{\tau} \\[-2mm]
&\qquad  
+ L(\norm{\sdiff{W_0}{\xi}}_{L^\infty} +2\norm{\profile}_{L^\infty})\ \integrall{0}{t}{ (1+\tau)^\beta \|W\|_{L^\infty} \|\partial_\xi W\|^2_{L^{2}}}{\tau}.
\end{align*}
For this inequality, we take $\beta$ and $\epsilon$ which satisfy
$$
\beta-\frac{1+\alpha}{\alpha} > 1, \qquad \frac32 \cos \bigPar{ \theta \tfrac{\pi}{2} } - \epsilon \beta > 0,
$$
obtaining
\begin{align*}
 \mathcal{E}_{\beta}(t)^2 &- \integrall{0}{t}{ (1+\tau)^\beta \integral{\R}{ f''(\profile) \profile' W^2 }{\xi} }{\tau} 
    + c \integrall{0}{t}{ \mathcal{D}_\beta(\tau)^2 }{\tau}  \\[-2mm]
 &\le  C(\|W_0\|_{H^1}^2 + \|W_0\|_{W^{1,1}})^2 \ (1+t)^{\beta-1/\alpha} \\
 & \qquad + L(\norm{\sdiff{W_0}{\xi}}_{L^\infty} +2\norm{\profile}_{L^\infty})\ \integrall{0}{t}{ (1+\tau)^\beta \|W\|_{L^\infty} \|U\|^2_{L^{2}}}{\tau} \ ,
\end{align*}
for some positive constant~$c$.
Finally, using \eqref{interpol}, the estimates~\eqref{est:W:Linfinity}--\eqref{est:U:Linfinity} and the smallness of $\norm{W_0}_{W^{1,1}}$, we arrive at
\begin{multline*}
 \mathcal{E}_{\beta}(t)^2  - \integrall{0}{t}{ (1+\tau)^\beta \integral{\R}{ f''(\profile) \profile' W^2 }{\xi} }{\tau}
   + c \integrall{0}{t}{ \mathcal{D}_\beta(\tau)^2 }{\tau} \\[-1mm]
 \le C(\|W_0\|_{H^1}^2 + \|W_0\|_{W^{1,1}})^2 (1+t)^{\beta-1/\alpha}  
 \le C E_1^2 \ (1+t)^{\beta-1/\alpha} 
\end{multline*}
and the desired estimate \eqref{decay_est_W}.
\end{proof}

\appendix
\section{Riesz-Feller operators} \label{app:RF}
To study the existence of traveling wave solutions with smooth profiles,
 we need the singular integral representation of Riesz-Feller operators~$\RieszFeller$.

\begin{proposition}[{\cite[Proposition 2.3]{Achleitner+Kuehn:2015}}] \label{prop:RieszFeller:extension}
If $1<\alpha<2$ and $|\theta|\leq \min\{\alpha,2-\alpha\}$,
 then for all $v\in\SchwartzTF(\R)$ and $x\in\R$ 
 \begin{multline} \label{eq:RieszFeller1b}
 \RieszFeller v(x) = c_1 \integrall{0}{\infty}{ \frac{v(x+\xi)-v(x)-v'(x)\,\xi}{\xi^{1+\alpha}} }{\xi}
  + c_2 \integrall{0}{\infty}{ \frac{v(x-\xi)-v(x)+v'(x)\,\xi}{\xi^{1+\alpha}} }{\xi} \,,
 \end{multline}
 for some constants $c_1, c_2 \geq 0$ with $c_1+c_2 >0$.
\end{proposition}

The singular integral representation~\eqref{eq:RieszFeller1b} for Riesz-Feller operators~$\RieszFeller$ is well-defined for $C^2_b$ functions
 such that $\RieszFeller C^2_b(\R)\subset C_b(\R)$.
\begin{proposition} \label{prop:RieszFeller:estimate}
The integral representation~\eqref{eq:RieszFeller1b}
 of $\RieszFeller$ with $1<\alpha<2$ and $|\theta|\leq \min\{\alpha,2-\alpha\}$
 is well-defined for functions $v\in C^2_b(\R)$ with
 \begin{equation} 
 \label{eq:estimate:RieszFeller}
  \sup_{x\in\R} |\RieszFeller v(x)|
      \leq \tfrac12 (c_1 +c_2) \|{v''}\|_{C_b(\R)} \frac{M^{2-\alpha}}{2-\alpha}
      + 2 (c_1 +c_2) \|v'\|_{C_b(\R)} \frac{M^{1-\alpha}}{\alpha-1} < \infty
 \end{equation}
 for some positive constant $M$
 and the positive constants $c_1$ and $c_2$ in Proposition~\ref{prop:RieszFeller:extension}.

Moreover, if $v\in C^2_b(\R)$ is a function such that the limits $\lim_{x\to\pm\infty} v(x)$ exist,
 then $\integral{\R}{ \RieszFeller v(x) }{x} = 0$.
\end{proposition}
\begin{proof}
The first statement follows by direct estimates on the extension of Riesz-Feller operators in~\eqref{eq:RieszFeller1b},
 see~\cite[Proposition 2.4]{Achleitner+Kuehn:2015}.
To prove the second statement,
 we consider the two summands in~\eqref{eq:RieszFeller1b} separately,
 starting with $\integrall{0}{\infty}{ \frac{v(x+\xi)-v(x)-v'(x)\xi}{\xi^{1+\alpha}} }{\xi}$ for any $v\in C^2_b(\R)$.
Like before, we rewrite the integral
 \begin{align*}
   \integrall{0}{\infty}{ \frac{v(x+\xi)-v(x)-v'(x)\xi}{\xi^{1+\alpha}} }{\xi} 
     &= \integrall{0}{\infty}{ \frac1{\xi^{1+\alpha}} \bigg[ \integrall{0}{1}{v'(x+\theta\xi)\, \xi}{\theta} -v'(x)\xi \bigg] }{\xi} \\
     &= \integrall{0}{\infty}{ \frac1{\xi^{\alpha}} \integrall{0}{1}{\big[ v'(x+\theta\xi) -v'(x) \big]}{\theta} }{\xi} \\
     &= \integrall{0}{\infty}{ \frac1{\xi^{\alpha}} \sdiff{}{x} \integrall{0}{1}{\big[ v(x+\theta\xi) -v(x) \big]}{\theta} }{\xi} \\ 
     &= \sdiff{}{x} \integrall{0}{\infty}{ \frac1{\xi^{\alpha}} \integrall{0}{1}{\big[ v(x+\theta\xi) -v(x) \big]}{\theta} }{\xi} \,,
 \end{align*}
 where exchanging integration and taking derivatives is possible,
 since in each step the integrands are absolutely integrable uniformly with respect to~$x$.
Moreover, 
 \begin{multline*}
  \integral{\R}{ \integrall{0}{\infty}{ \frac{v(x+\xi)-v(x)-v'(x)\xi}{\xi^{1+\alpha}} }{\xi} }{x}
   = \integral{\R}{ \sdiff{}{x} \integrall{0}{\infty}{ \frac{1}{\xi^{\alpha}} \integrall{0}{1}{\big[ v(x+\theta\xi) -v(x) \big]}{\theta} }{\xi} }{x}
 \end{multline*}
 and the primitive satisfies 
 \begin{multline*} 
 \lim_{x\to\pm\infty} \integrall{0}{\infty}{ \frac1{\xi^{\alpha}} \integrall{0}{1}{\big[ v(x+\theta\xi) -v(x) \big]}{\theta} }{\xi} \\
   =  \integrall{0}{\infty}{ \frac1{\xi^{\alpha}} \integrall{0}{1}{\lim_{x\to\pm\infty}\big[ v(x+\theta\xi) -v(x) \big]}{\theta} }{\xi} 
   = 0 \,,
 \end{multline*}
 where exchanging integration and taking limits is possible,
 since in each step the integrands are absolutely integrable 
 and $\lim_{x\to\pm\infty}\big[ v(x+\theta\xi) -v(x) \big]=0$ due to the assumptions on~$v$.
\end{proof}

Using the singular integral representation of $\RieszFeller$ and \cite[Lemma~1]{Droniou+Imbert:2006},
 we deduce the following result:
\begin{lemma}
 \label{lemma:CI} 
  Let $1<\alpha<2$, $u\in C^2_b(\R)$ and $\eta\in C^2(\R)$ be a convex function.
  Then $\eta'(u) ( \RieszFeller u ) \leq \RieszFeller \eta(u)$.
\end{lemma}
\begin{proof}
Since $\eta$ is convex, we have $\eta'(a) (b-a) \leq \eta(b)-\eta(a)$. Hence,
\[ \eta'(u(x)) (u(x+z)-u(x)) \leq \eta(u(x+z)) - \eta(u(x)) \]
and $\eta'(u(x)) (u(x+z) -u(x) -u'(x)\cdot z) \leq \eta(u(x+z)) - \eta(u(x)) -(\eta(u))'(x)\cdot z$. 
The conclusion follows from these inequalities and Equation~\eqref{eq:RieszFeller1b}.
\end{proof}


\medskip \noindent
\textbf{Acknowledgements:}
The first author was partially supported by Austrian Science Fund (FWF) under grant P28661 and the FWF-funded SFB \# F65.

\def\cprime{$'$} \def\cprime{$'$} \def\cprime{$'$} \def\cprime{$'$}
  \def\cprime{$'$} \def\cprime{$'$} \def\cprime{$'$}

\end{document}